\documentclass[leqno]{article}
\usepackage{amsfonts}
\usepackage{eurosym}
\usepackage{amsmath}
\usepackage{amsmath}
\usepackage{amssymb}
\usepackage{graphicx}
\providecommand{\U}[1]{\protect\rule{.1in}{.1in}}
\providecommand{\U}[1]{\protect\rule{.1in}{.1in}}
\providecommand{\U}[1]{\protect\rule{.1in}{.1in}}
\providecommand{\U}[1]{\protect\rule{.1in}{.1in}}
\providecommand{\U}[1]{\protect\rule{.1in}{.1in}}
\providecommand{\U}[1]{\protect\rule{.1in}{.1in}}
\providecommand{\U}[1]{\protect\rule{.1in}{.1in}}
\newtheorem{theorem}{Theorem}

\newtheorem{proposition}[theorem]{Proposition}
\newtheorem{remark}[theorem]{Remark}

\newenvironment{proof}[1][Proof]{\noindent\textbf{#1.} }{\ \rule{0.5em}{0.5em}}
\begin{document}

\title{\textbf{Four-dimensional naturally
reductive pseudo-Riemannian spaces
}}
\author{\textsc{W. Batat, M. Castrill\'on L\'opez
and E. Rosado Mar\'\i a}}
\date{}
\maketitle

\begin{abstract}
The classification of 4-dimensional naturally reductive
pseudo-Rieman\ nian spaces is given. This classification comprises
symmetric spaces, the product of 3-dimensional naturally reductive
spaces with the real line and new families of indecomposable
manifolds which are studied at the end of the article. The
oscillator group is also analyzed from the point of view of this
classification.

\end{abstract}

\noindent The three authors have been partially supported by
MICINN, Spain, under grant MTM2011-22528.

\medskip


\medskip


\section{Introduction}

Homogeneous manifolds play a preeminent role in Differential
Geometry and have deserved thorough studies and classifications
from different perspectives. Among these spaces, naturally
reductive manifolds are possibly the simplest class besides the
class of Lie groups or symmetric spaces. This is probably due to
the fact that they generalize these spaces in a simple way.
Classifications of low dimensional naturally reductive Riemannian
homogeneous manifolds can be found in classical references. Beyond
the trivial result in surfaces, all connected and simply connected
3-dimensional naturally homogeneous spaces are give in \cite{TV}:
they comprise symmetric spaces together with the Lie groups
$SU(2)$, $\widetilde{\mathrm{SL}(2,\mathbb{R})}$ and the
Heisenberg group, endowed with convenient left invariant metrics.
The four dimensional case is tackled in \cite{KowVan1} where it is
proved that under the same topological conditions, a naturally
reductive Riemannian 4-manifold necessarily splits as a product of
a 3-dimensional naturally reductive manifold and $\mathbb{R}$. We
have to wait for the 5-dimensional case to get new indecomposable
naturally reductive manifolds (see \cite{KowVan2}).

The study of naturally reductive pseudo-Riemannian spaces also
deserves special attention. The classification in the
3-dimensional setting has been recently obtained in  \cite{Cal}, \cite{Bat}
where, again, the manifold is either symmetric, $SU(2)$,
$\widetilde{\mathrm{SL}(2,\mathbb{R})}$ or the Heisenberg group
with convenient metrics. The four dimensional case has attired
much interest in the literature (see for example \cite{Barco},
\cite{Ovando} where the structure of naturally reductive groups
are analyzed) probably because of the possible connections of
these spaces with plausible relativistic models. The goal of this
paper is to provide the complete classification of 4-dimensional
naturally reductive pseudo-Riemannian manifolds of $(1,3)$ or
$(2,2)$ signatures. Surprisingly, the main results (see Theorem
\ref{Loren} and \ref{dosdos}) show that, besides the product of a
3-dimensional naturally reductive manifold and $\mathbb{R}$, there
is a family of indecomposable manifolds. This situation has no
counterpart in the Riemannian case.

The structure of the article is as follows. We first review the
basic concepts and properties of naturally reductive manifolds,
specially those connected with the notion of homogeneous structure
tensors. We then follow the technique of Kowalski and Vanhecke,
although we cannot simply generalize \cite{KowVan1} due to the
existence of the new families mentioned above. At the end of the
article, we explore the geometry of these new manifolds to be sure
that they are indecomposable and non-symmetric. Finally, we apply
Theorem \ref{Loren} to the analysis of the 4-dimensional
oscillator group, probably the most relevant naturally reductive
Lorentzian example in the literature. We give a decomposition of
this space which is different to the one of its traditional
definition.

\section{Preliminaries}

\subsection{Naturally reductive spaces}

Let $(M,g)$ be a reductive pseudo-Riemannian homogeneous manifold
of dimension $n$. This means that $M=G/H$, where $G$ is connected
Lie group of isometries acting transitively and effectively on
$M$, $H$ is the isotropy of a point $o\in M$, and the Lie algebra
$\mathfrak{g}$ of $G$ admits a decomposition
$\mathfrak{g}=\mathfrak{h}\oplus\mathfrak{m}$ such that
$[\mathfrak{h,m}]\subset\mathfrak{m}$, where $\mathfrak{h}$ is the
Lie algebra of $H$. The mapping $A\mapsto A_{o}^{\ast} =
d/d\varepsilon|_{\varepsilon=0}\exp(\varepsilon A)\cdot o$ defines
an isomorphism between $\mathfrak{m}$ and $T_{o}M$ which, in
addition, is used to transfer the metric $g$ to $\mathfrak{m}$.
For convenience, along the article we will denote both the metric
in $T_oM$ and in $\mathfrak{m}$ by $\langle \cdot , \cdot
\rangle$. The decomposition of $\mathfrak{g}$ is said to be
naturally reductive if in addition
\begin{equation}
\langle\lbrack X,Y]_{\mathfrak{m}},Z\rangle
+\langle\lbrack X,Z]_{\mathfrak{m}},Y\rangle =0
\text{ for }X,Y,Z\in\mathfrak{m},
\label{NaturallyReductiveCondition}
\end{equation}
where $[\cdot,\cdot]_{\mathfrak{m}}$ is the $\mathfrak{m}$-part of
the bracket (see, \emph{e.g.}, \cite[Chapter X, section 3]{KN},
\cite[Chapter 11, Definition 23]{ONeill}). Let
$\widetilde{\nabla}$ be the canonical connection of the reductive
homogenous space $M=G/H$. It is well known that the torsion tensor
$\tilde{T}$ and the curvature tensor $\tilde{R}$ of
$\widetilde{\nabla}$ at the point $o$ read
\begin{equation}
\widetilde{T}(X,Y)_{o}=-[X,Y]_{\mathfrak{m}},
\quad\widetilde{R}(X,Y)_{o}=-[X,Y]_{\mathfrak{h}},
\quad\forall X,Y\in\mathfrak{m}.
\label{1}
\end{equation}
Recalling that $G$-invariant tensor fields on $M$
are parallel with respect to the connection
$\widetilde{\nabla}$, we have
\begin{equation}
\widetilde{\nabla}g=\widetilde{\nabla}\widetilde{R}
=\widetilde{\nabla }\widetilde{T}=0.
\label{2}
\end{equation}
Conditions (\ref{1}) and (\ref{2}) provide
interesting properties. First, the subalgebra
$\mathfrak{k}\subset\mathfrak{h}$ generated by
all projections
$[X,Y]_{\mathfrak{h}}=-\tilde{R}(X,Y)$,
$X,Y\in\mathfrak{m}$, belongs to the holonomy
algebra and hence its elements $A\in\mathfrak{k}$
act as derivation on the tensor algebra of $\mathfrak{m}$ and
\begin{equation*}
A\cdot g=A\cdot\widetilde{R}=A\cdot\widetilde{T}=0.
\label{3}
\end{equation*}
Second, the Bianchi identities
(see, \cite[Chapter III, Theorem 5.3]{KN}) become
\begin{align}
\mathfrak{S}_{X,Y,Z}\tilde{R}(X,Y)Z &
=\mathfrak{S}_{X,Y,Z}\tilde{T}(\tilde{T}(X,Y),Z),
\label{Bianchi1}\\
\mathfrak{S}_{X,Y,Z}\tilde{R}(\tilde{T}(X,Y),Z) &  =0,
\label{Bianchi2}
\end{align}
for all $X,Y,Z\in\mathfrak{m}$, where
$\mathfrak{S}_{X,Y,Z}$ denotes the
cyclic sum with respect to $X,Y,Z$.

With both tensors $\tilde{T}$ and $\tilde{R}$
we can recover two important objects.
On one hand, the Riemann curvature tensor
$R$ defined by the Levi-Civita connection
$\nabla$ at $T_{o}M$ satisfies the formula
\begin{equation}
R(X,Y)=\tilde{R}(X,Y)+[D_{X},D_{Y}]+D_{\tilde{T}(X,Y)},
\label{R}
\end{equation}
where $D_{X}$ denotes the difference
$(1,1)$-tensor $D_{X}=\nabla _{X}-\widetilde{\nabla}_{X}$,
which from (\ref{NaturallyReductiveCondition}) and (\ref{1}) is
\begin{equation}
D_{X}Y=-\tfrac{1}{2}\tilde{T}(X,Y).
\label{D}
\end{equation}
On the other hand (see \cite[Chapter 1, (1.79)]{TV}),
the brackets of the Lie algebra
$\mathfrak{g}=\mathfrak{m}\oplus\mathfrak{h}$
are defined as
\begin{equation}
\left\{
\begin{array}
[c]{ll}
\lbrack U,V]=UV-VU, & U,V\in\mathfrak{h},\\
\lbrack U,X]=U(X), & U\in\mathfrak{h},~X\in\mathfrak{m},\\
\lbrack X,Y]=-\tilde{T}(X,Y)+\widetilde{R}(X,Y), & X,Y\in\mathfrak{m}.
\end{array}
\right.
\label{bracket}
\end{equation}

\begin{remark}
A homogeneous structure tensor in a pseudo-Riemannian manifold
$(M,g)$ is a $(1,1)$-tensor $D$ satisfying (\ref{2}) for the
connection $\tilde{\nabla }=\nabla-D$ where $\nabla$ is the
Levi-Civita connection. By a classical result of Ambrose and
Singer (see \cite{AS}) a connected, simply connected and complete
manifold is reductive homogeneous if and only if it has a
homogenous structure tensor. The set of these tensors are
classified in three primitive classes invariant under the action
of the orthogonal group of the appropriate signature (see
\cite{Gadea0}). Tensors $D$ belonging to the class
$\mathcal{T}_{2}\oplus\mathcal{T}_{3}$ are those satisfying the
property $D_{X}Y+D_{Y}X=0$ and characterize natural reductivity.
The tensor in (\ref{D}) it is obviously in
$\mathcal{T}_2\oplus\mathcal{T}_3$.
\end{remark}

\begin{proposition}
\label{2p}If a pseudo-Riemannian manifold
$(M,g)$ admits the null tensor as a homogeneous
structure tensor, then it is locally symmetric.

If a naturally reductive homogeneous manifold
$(M,g)$ has null intrinsic curvature
$\tilde{R}$, then it is locally symmetric.
\end{proposition}

\begin{proof}
If the homogeneous structure tensor $D$ vanishes,
then $\nabla R=0$ and the manifold is locally
symmetric. Similarly, if $\tilde{R}=0$, then $H$ is
discrete. The universal covering of $M$ is a
Lie group with an invariant metric satisfying
$g([X,Y],Z)+g([X,Z],Y)=0$, which is necessarily symmetric.
\end{proof}

\subsection{Decomposition of manifolds}

We now recall the following classical results.

\begin{theorem}
\label{Rham-Wu}\textbf{(de Rham-Wu decomposition) }
Let $(M,g)$ be a simply connected and complete
pseudo-Riemannian manifold. Then $(M,g)$ can be
decomposed as a pseudo-Riemannian product
\[
(M,g)\simeq(M_1,g_1)\times\cdots\times(M_k,g_k)
\]
where for each $(M_i,g_i)$ and any $x_i\in M_i$,
the tangent space $T_{x_i}M_i$ does not admit a
proper non-degenerate subspace, invariant
with respect to the holonomy. The decomposition above
is unique up to order of the factors. Moreover,
the connected components of the identity of the
isometry groups satisfy
\[
\mathcal{I}^0(M,g)
=\mathcal{I}^0(M_1,g_1)
\times\cdots\times
\mathcal{I}^0(M_k,g_k).
\]

\end{theorem}

The proof of this result can be found,
for example, in \cite{Wu}. As a
consequence of this result we have:

\begin{proposition}
Let $(M,g)$ be a pseudo-Riemannian manifold and
\[
(M,g)\simeq(M_1,g_1)\times\cdots\times(M_k,g_k)
\]
its de Rham-Wu decomposition. Then $(M,g)$
is a naturally reductive homogeneous space
if and only if each $(M_i,g_i)$
is a naturally reductive homogeneous space.
\end{proposition}

\begin{proposition}
\label{descomposicion}Let $(M,g)$ be a connected
and simply connected naturally reductive homogeneous
pseudo-Riemannian manifold and let $o\in M$.
Suppose that $T_{o}M=W\oplus W^{\perp}$ and that
\begin{align}
\tilde{T}(\pi_iX,\pi_iY)  &
=\pi_i\tilde{T}(X,Y),
\label{ProyT}\\
\tilde{R}(\pi_iX,\pi_iY)\pi_iZ  &
=\pi_i\tilde{R}(X,Y)Z,
\label{ProyR}
\end{align}
for all $X,Y,Z\in T_{o}M$, where
$\pi_{1}\colon T_{o}M\rightarrow W$ and
$\pi_{2}\colon T_{o}M\rightarrow W^{\perp}$
are the natural projections. Then $M$ is
the pseudo-Riemannian product of two
naturally reductive homogeneous
pseudo-Riemannian manifolds.
\end{proposition}

This result is proved for the Riemannian case
in \cite{KowVan2}. The proof for the case of
arbitrary signature is similar, with the only
difference that one has to ensure the non
degeneracy of the restriction of the metric $g$
to $W\subset T_{o}M$, in order to apply the
de Rham-Wu Theorem. This condition is
satisfied as $T_{o}M=W\oplus W^{\perp}$.
We also have the following result:

\begin{proposition}
\label{Decomposition} If
$W=\mathrm{span}\{\tilde{T}(X,Y)~|~X,Y\in T_{o}M\}$
is a proper non-degenerate space then the conditions
of Proposition \ref{descomposicion} are satisfied
and the manifold $M$ is decomposable.
\end{proposition}

\begin{proof}
From (\ref{NaturallyReductiveCondition}) we have
$\tilde{T}(X,Y)=0$ if $Y\in W^{\perp}$ and
therefore (\ref{ProyT}) is satisfied.

We now check that condition (\ref{ProyR})
is also satisfied. If
$X=\tilde{T}(U,V)\in W$, $Z\in W^{\perp}$,
using (\ref{Bianchi2}) with $U,V,Z$ gives
$\tilde{R}(X,Z)=0$. Now, if $X,Y\in W$ and
$Z\in W^{\perp}$ from (\ref{Bianchi1}) we get
$\tilde{R}(X,Y)Z=0$ analogously $\tilde{R}(X,Y)Z=0$
for $X,Y\in W^{\perp}$, $Z\in W$. Finally,
taking (\ref{R}) into account we get
$\langle R(X,Y)Z,U\rangle=\langle\tilde{R}(X,Y)Z,U\rangle$
and $\langle R(X,Y)U,Z\rangle=\langle\tilde{R}(X,Y)U,Z\rangle$
for any $U\in W^{\perp}$ and from the symmetries of
the Riemann curvature tensor we have
$\langle\tilde{R}(X,Y)Z,U\rangle=0$ for any
$X,Y,Z\in W$ and hence $\tilde{R}(X,Y)Z\in W$.

If $X,Y,Z\in W^{\perp}$, $U\in W$, from (\ref{R})
we have
$\langle R(X,Y)Z,U\rangle=\langle\tilde{R}(X,Y)Z,U\rangle$, and
$\langle R(Z,U)X,Y\rangle=\langle\tilde{R}(Z,U)X,Y\rangle$.
From the symmetries of the Riemann curvature tensor
we have
$\langle\tilde{R}(X,Y)Z,U\rangle
=\langle\tilde{R}(Z,U)X,Y\rangle$
but the right hand side vanishes. Hence
$\langle\tilde{R}(X,Y)Z,U\rangle=0$, $\forall U\in W$
and then $\tilde{R}(X,Y)Z\in W^{\perp}$.
\end{proof}

\subsection{Normal forms of skew-adjoint operators}

Let $(V,\langle\cdot,\cdot\rangle)$ be a
metric vector space and let $A\colon V\rightarrow V$
be a skew-symmetric linear endomorphism, that is an
endomorphism satisfying
\[
\langle A(u),v\rangle=-\langle u,A(v)\rangle,
\qquad\forall u,v\in V.
\]
If $A(W)\subset W$ for a subspace $W\subset V$
for which the restriction of the metric is
non-degenerate, then $A(W^{\perp})\subset W^{\perp}$
and we can decompose $V=W\oplus W^{\perp}$.
In this case, the endomorphism $A$ is said to
be reducible. If there is no such an invariant
non-degenerate subspace $W$, we say that $A$ is irreducible.

\begin{proposition}
\label{Lorentz}Let $(V,\langle\cdot,\cdot\rangle)$
be a $4$-dimensional Lorentzian vector space.
For any skew-symmetric endomorphism $A\colon V\rightarrow V$,
there exists an orthonormal basis $\mathcal{B}$ of $V$ with
respect to which the matrix of $\langle\cdot,\cdot\rangle$
is $\mathrm{diag}(-1,1,1,1)$ and the matrix of
$A$ is one of the following types:

\begin{description}
\item[a)] $A_1=\pm\left(
\begin{array}
[c]{rrrr}
0 & 0 & 1 & 0\\
0 & 0 & 1 & 0\\
1 & -1 & 0 & 0\\
0 & 0 & 0 & 0
\end{array}
\right)  $,

\item[b)] $A_{\alpha\beta}
=\alpha A_{2}+\beta A_{3}$, with
$\alpha,\beta \in\mathbb{R}$ and
\[
A_2=\left(
\begin{array}
[c]{rrrr}
0 & 1 & 0 & 0\\
1 & 0 & 0 & 0\\
0 & 0 & 0 & 0\\
0 & 0 & 0 & 0
\end{array}
\right)  ,\qquad A_{3}=\left(
\begin{array}
[c]{rrrr}
0 & 0 & 0 & 0\\
0 & 0 & 0 & 0\\
0 & 0 & 0 & 1\\
0 & 0 & -1 & 0
\end{array}
\right)  .
\]

\end{description}
\end{proposition}

\begin{proof}
From \cite{SS}, every skew-symmetric
transformation in a $4$-dimensional
manifold is reducible. We then have
$V=W\oplus W^{\perp}$ where $W$ is
Lorentzian and $W^{\perp}$ Riemannian.
We consider that $A|_{W}$ is irreducible.

If $\dim W=1$, then $A|_{W}=(0)$ and
there is an orthonormal basis in $W^{\perp}$
such that $A|_{W^{\perp}}$ is a Riemannian skew-symmetric
endomorphism. We then get that $A$ is as in the case b)
with $\alpha=0$.

If $\dim W=2$, then (see \cite{SS}) there are
orthonormal basis in $W$ and $W^{\perp}$
for which the matrices of $A|_{W}$ and $A|_{W^{\perp}}$
are respectively
\[
\left(
\begin{array}
[c]{cc}
0 & \alpha\\
\alpha & 0
\end{array}
\right)  ,\alpha\neq0,\qquad\left(
\begin{array}
[c]{cc}
0 & \beta\\
-\beta & 0
\end{array}
\right)  ,\beta\in\mathbb{R},
\]
and we recover the matrices in the case b)
with $\alpha\neq0$.

If $\dim W=3$, then (see \cite{SS})
there is a basis such that $A|_{W}$
defines a matrix as the top left $3\times 3$
submatrix in a), so that the proof is complete.
\end{proof}

\begin{proposition}
\label{Neutral} Let $(V,\langle\cdot,\cdot\rangle)$
be a $4$-dimensional vector space with a
$(2,2)$-signature metric. Let $A\colon V\rightarrow V$
be a skew-symmetric endomorphism.
Then we have:

\noindent If $A$ is reducible, then
there is an orthonormal basis $\mathcal{B}$ of $V$
with respect to which the matrix of
$\langle\cdot ,\cdot\rangle$ is
$\mathrm{diag}(-1,-1,1,1)$ and the matrix of $A$
is one of the following
\begin{align*}
\text{\emph{a1)} }A_{1}  &  =\left(
\begin{array}
[c]{rrrr}
0 & 0 & 0 & 0\\
0 & 0 & 0 & 1\\
0 & 0 & 0 & 1\\
0 & 1 & -1 & 0
\end{array}
\right)  ,
\quad\text{\emph{a2)} }A_2=\left(
\begin{array}
[c]{rrrr}
0 & \alpha & 0 & 0\\
-\alpha & 0 & 0 & 0\\
0 & 0 & 0 & \beta\\
0 & 0 & -\beta & 0
\end{array}
\right)  ,
\text{ }\beta\neq0,\\
\text{\emph{a3)} }A_3  &  =\left(
\begin{array}
[c]{cccc}
0 & 0 & \beta & 0\\
0 & 0 & 0 & \alpha\\
\beta & 0 & 0 & 0\\
0 & \alpha & 0 & 0
\end{array}
\right)  ,
\text{ }\alpha\neq 0.
\end{align*}
\noindent If $A$ is irreducible, then
there is a basis $\mathcal{B}$ of $V$
with respect to which the matrix of
$\langle\cdot,\cdot\rangle$ is
\[
\left(
\begin{array}
[c]{rrrr}
0 & 0 & 0 & -1\\
0 & 0 & 1 & 0\\
0 & 1 & 0 & 0\\
-1 & 0 & 0 & 0
\end{array}
\right)
\]
and the matrix of $A$ is one of the following
\begin{align*}
\text{\emph{b1)} }B_{1}  &  =\left(
\begin{array}
[c]{rrrr}
0 & -\nu & 1 & 0\\
\nu & 0 & 0 & 1\\
0 & 0 & 0 & -\nu\\
0 & 0 & \nu & 0
\end{array}
\right)  ,
\quad\text{\emph{b2) }}B_{2}=\left(
\begin{array}
[c]{rrrr}
\lambda & 0 & 1 & 0\\
0 & -\lambda & 0 & 1\\
0 & 0 & \lambda & 0\\
0 & 0 & 0 & -\lambda
\end{array}
\right)  ,
\ \lambda\neq0,\\
\text{\emph{b3) }}B_{3}  &  =\left(
\begin{array}
[c]{cccc}
\xi & \nu & 0 & 0\\
-\nu & \xi & 0 & 0\\
0 & 0 & -\xi & \nu\\
0 & 0 & -\nu & -\xi
\end{array}
\right)  ,\
\xi,\nu\neq0.
\end{align*}

\end{proposition}

\begin{proof}
The possible cases are obtained from the classification of
irreducible skew-symmetric endomorphisms in spaces with signature
$(2,n-2)$ given in \cite{Bou}or \cite[Theorem 4.1]{Leitner}.
\end{proof}

\section{Classification theorem}

\begin{theorem}
\label{Loren}Let $(M,g)$ be a simply connected
naturally reductive Lorentzian manifold
of dimension $4$. Then $M$ is either symmetric,
decomposable or isometric to $G/H$ with

\begin{enumerate}
\item $G=\widetilde{SL(2,\mathbb{R})}\times\mathbb{R}^2$
and $H$ a $1$-dimensional subgroup. If the
Lie algebra is spanned as
$\mathfrak{g}=$\textrm{span}
$\{Y_{1},Y_{2},Y_{3},T_{1},T_{2}\}$,
with non-null brackets
$[Y_{1},Y_{2}]=-\lambda Y_{3}$,
$[Y_{1},Y_{3}]=\lambda Y_{2}$,
$[Y_{2},Y_{3}]=Y_{1}$, then
$\mathfrak{h}=$\textrm{span}$\{A$\thinspace$\}$ and
$\mathfrak{m}=$\textrm{span}$\{X_1,X_2,X_3,X_4\}$
with
\begin{align*}
X_{1}  &
=\tfrac{1}{\lambda}Y_{1}
+\left(
1-\tfrac{1}{\lambda}
\right)  T_{1}
+\tfrac{1}{\lambda}T_{2},
\quad
X_{2}=\tfrac{1}{\lambda}Y_{1}
-\tfrac{1}{\lambda}T_{1}
+\left(
1+\tfrac{1}{\lambda}
\right)
T_{2},\\
X_{3}  &  =Y_{2},
\quad
X_{4}=Y_{3},\quad
A=\tfrac{1}{\lambda}Y_1-\tfrac{1}{\lambda}T_1
+\tfrac{1}{\lambda}T_2.
\end{align*}
The metric $g$ in $M$ is induced by the metric
$\mathrm{diag}(-1,1,1,1)$ in
$\mathfrak{m}$ by $G$ invariance.

\item $G$ belonging to the family of
simply connected Lie groups with Lie
algebra
$\mathfrak{g}=$\textrm{span}$\{X_1,X_2,X_3,X_4,A,B\}$
and structure constants
\[
\begin{array}
[c]{l}
\lbrack A,X_{1}]=-[A,X_{2}]=X_{3},\\
\lbrack B,X_{1}]=-[B,X_{2}]=X_{4},\\
\lbrack A,X_{3}]=[B,X_{4}]=X_{1}+X_{2},\\
\lbrack X_{1},X_{3}]=-[X_{2},X_{3}]=-cX_{4}+\alpha A+\beta B,\\
\lbrack X_{1},X_{4}]=-[X_{2},X_{4}]=cX_{3}+\beta A+\delta B,\\
\lbrack X_{3},X_{4}]=c\left(  X_{1}+X_{2}\right)  ,
\end{array}
\]
with $c$, $\alpha\,$, $\beta$, $\delta\in\mathbb{R}$.
The Lie subalgebra of $H$ is
$\mathfrak{h}=$\textrm{span}$\{A,B$\thinspace$\}$.
The metric $g$ in $M$ is induced by the metric
$\mathrm{diag}(-1,1,1,1)$ in the complement
$\mathfrak{m}=$\textrm{span}$\{X_{1},X_{2},X_{3},X_{4}\}$
by $G$ invariance.
\end{enumerate}
\end{theorem}

\begin{theorem}
\label{dosdos}Let $(M,g)$ be a simply connected
naturally reductive $(2,2)$-signature pseudo-Riemannian
manifold of dimension $4$. Then $M$ is either symmetric,
decomposable or isometric to $G/H$ with

\begin{enumerate}
\item $G=\widetilde{SL(2,\mathbb{R})}\times\mathbb{R}^{2}$
and $H$ a $1$-dimensional subgroup. If the Lie algebra
is spanned as
$\mathfrak{g}=$\textrm{span}$\{Y_1,Y_2,Y_3,T_1,T_2\}$,
with non-null brackets
$[Y_1,Y_2]=\lambda Y_3$,
$[Y_1,Y_3]=\lambda Y_2$,
$[Y_2,Y_3]=Y_1$, then
$\mathfrak{h}=$\textrm{span}$\{A$\thinspace$\}$ and
$\mathfrak{m}=$\textrm{span}$\{X_1,X_2,X_3,X_4\}$
with
\begin{align*}
X_{1}  &  =-\tfrac{1}{\lambda}Y_{1}
+
\left(
1+\tfrac{1}{\lambda}
\right)
T_{1}+\tfrac{1}{\lambda}T_{2},
\quad
X_{2}=Y_{2},
\quad
X_{4}=Y_{3},\\
X_{3}  &  =\tfrac{1}{\lambda}Y_{1}
-\tfrac{1}{\lambda}T_{1} +
\left(
1-\tfrac{1}{\lambda}
\right)  T_{2},
\quad
A=\tfrac{1}{\lambda}Y_{1}
-\tfrac{1}{\lambda}T_{1}
-\tfrac{1}{\lambda}T_{2}.
\end{align*}
The metric $g$ in $M$ is induced by the metric
$\mathrm{diag}(-1,-1,1,1)$ in
$\mathfrak{m}$ by $G$ invariance.$\allowbreak$

\item $G$ belonging to the family of simply
connected Lie groups with Lie algebra
$\mathfrak{g}=$\textrm{span}$\{X_1,X_2,X_3,X_4,A,B\}$
and structure constants
\[
\begin{array}
[c]{l}
\lbrack A,X_{2}]=-[A,X_{3}]=X_{4},\\
\lbrack B,X_{1}]=[A,X_{4}]=X_{2}+X_{3},\\
\lbrack B,X_{2}]=-[B,X_{3}]=-X_{1},\\
\lbrack X_{1},X_{2}]=-[X_{1},X_{3}]=-cX_{4}-\beta A-\delta B,\\
\lbrack X_{2},X_{4}]=-[X_{3},X_{4}]=cX_{1}-\alpha A+\beta B,\\
\lbrack X_{1},X_{4}]=-c\left(  X_{2}+X_{3}\right)  ,
\end{array}
\]
with $c$, $\alpha\,$, $\beta$, $\delta\in\mathbb{R}$.
The Lie subalgebra of $H$ is
$\mathfrak{h}=$\textrm{span}$\{A,B$\thinspace$\}$.
The metric $g$ in $M$ is induced by the metric
$\mathrm{diag}(-1,-1,1,1)$ in the complement
$\mathfrak{m}=$\textrm{span}$\{X_1,X_2,X_3,X_4\}$
by $G$ invariance.
\end{enumerate}
\end{theorem}

\begin{remark}
\label{remi}The families of algebras in Theorem \ref{Loren}-2 and
Theorem \ref{dosdos}-2 include some particular cases where $M=G/H$
is symmetric. That happens when $c=0$ and in the cases studied in
Proposition \ref{Case 1} and \ref{Case 2}. In addition, when
$\beta=\delta=0$, the quotient $M=G/H$ is the same as
$G^{\prime}/H^{\prime}$, where the Lie algebras are
$\mathfrak{g}^{\prime}=$\textrm{span}$\{X_1,X_2,X_3,X_4,A\}$ and
$\mathfrak{h}^{\prime}=$\textrm{span}$\{A$\thinspace$\}$. In any
case, in general, one can check that the structure constants of
these families define a solvable $6$-dimensional Lie algebra with
$5$-dimensional non-Abelian nilradical. It thus belongs to the
list of all possible algebras with these properties appearing in
\cite[Table 13]{Campoamor}, \cite[Table 3]{FSH}. Some computations
show that the dependence of $\mathfrak{g}$ on the parameters gives
different cases of the aforementioned list.
\end{remark}

\section{Proof of Theorem \ref{Loren}}

Let $\left(  X_1,X_2,X_3,X_4\right)  $
be an orthonormal basis in $T_oM$ such that
$\langle X_{i},X_{j}\rangle=\varepsilon_{i}\delta_{ij}$,
with $\varepsilon_{1}=-1$, $\varepsilon_{i}=1$
for $i=2,3,4$. We write:
\begin{equation}
\tilde{T}(X_{i},X_{j})
=\tilde{T}_{ij}^{k}X_{k},
\quad i,j,k=1,\ldots,4.
\label{T_tilda}
\end{equation}
From (\ref{NaturallyReductiveCondition}),
we have:
$\varepsilon_{k}\tilde{T}_{ij}^{k}
+\varepsilon_{j}\tilde{T}_{ik}^{j}=0$,
$i,j,k=1,\ldots,4$. Hence
$\tilde{T}_{ij}^{i}=0$ for $1\leq i\leq j\leq4$,
and by denoting
$\tilde{T}_{12}^{3}=a$,
$\tilde{T}_{12}^{4}=b$,
$\tilde{T}_{13}^{4}=c$,
$\tilde{T}_{23}^{4}=d$, we have:
\begin{equation}
\left\{
\begin{array}
[c]{l}
\tilde{T}(X_{1},X_{2})=aX_{3}+bX_{4},\\
\tilde{T}(X_{1},X_{3})=-aX_{2}+cX_{4},\\
\tilde{T}(X_{1},X_{4})=-bX_{2}-cX_{3},\\
\tilde{T}(X_{2},X_{3})=-aX_{1}+dX_{4},\\
\tilde{T}(X_{2},X_{4})=-bX_{1}-dX_{3},\\
\tilde{T}(X_{3},X_{4})=-cX_{1}+dX_{2}.
\end{array}
\right.
\label{TildaLorentz}
\end{equation}
We consider the skew-symmetric operator
$A=\tilde{R}(X,Y)$ for a choice of
$X,Y\in T_{o}M$. If $A=0$ for all choices
of $X,Y$, then $\mathfrak{h=}\{0\}$
and $M$ is symmetric (see Proposition \ref{2p}).
We thus assume $\tilde{T}\neq0$ and there exist
$X,Y\in T_{o}M$ such that $A=\tilde{R}(X,Y)\neq0$.
We use the classification of
Proposition \ref{Lorentz}.

\subsection{Case a)\label{Case a}}

Suppose that a curvature transformation
$A=\tilde{R}(X,Y)$ exists such that
\begin{equation}
AX_{1}=X_{3},
\quad
AX_{2}=-X_{3},
\quad
AX_{3}=X_{1}+X_{2},
\quad
AX_{4}=0,
\label{A_Lorentz_case_a}
\end{equation}
as in Proposition \ref{Lorentz}-(a) (for the opposite sign, just
consider $\tilde{R}(Y,X)$). By applying $A\cdot\tilde{T}=0$ to
$X_{i},X_{j}$ for $1\leq i<j\leq4$, and taking
(\ref{TildaLorentz}) and (\ref{A_Lorentz_case_a}) into account we
easily get: $b=0$ and $c+d=0$.

In the case, $a\neq0$, $c=0$ (resp. $a\cdot c\neq0$)
if we take
$W=\mathrm{span}\{X_1,X_2,X_3\}$
(resp. $W=\mathrm{span}\{X_3,-aX_2+cX_4,X_1+X_2\}$)
we conclude that $M$ is decomposable by virtue
of Proposition \ref{Decomposition}. We thus consider
$a=0$, $c\neq0$. From the Bianchi identities
(\ref{Bianchi1}), (\ref{Bianchi2}) and imposing
$A\cdot\widetilde{R}=0$, we obtain:
\[
\left\{
\begin{array}
[c]{l}
\tilde{R}(X_{1},X_{2})
=\tilde{R}(X_{3},X_{4})=0,\\
\tilde{R}(X_{1},X_{3})
=\tilde{R}_{133}^{2}A+\tilde{R}_{143}^{2}B,\\
\tilde{R}(X_{1},X_{4})
=\tilde{R}_{143}^{2}A+\tilde{R}_{144}^{2}B,\\
\tilde{R}(X_{2},X_{j})
=-\tilde{R}(X_{1},X_{j}),
\quad
j=3,4,
\end{array}
\right.
\]
where
\[
B=\left(
\begin{array}
[c]{cccc}
0 & 0 & 0 & 1\\
0 & 0 & 0 & 1\\
0 & 0 & 0 & 0\\
1 & -1 & 0 & 0
\end{array}
\right)  .
\]
We have the following possibilities:

\begin{itemize}
\item If
$(\tilde{R}_{143}^{2})^{2}\neq\tilde{R}_{133}^{2}\tilde{R}_{144}^{2}$
then
$\mathfrak{h} =\mathrm{span}\{\tilde{R}(X,Y)|X,Y\in\mathfrak{m}\}
=\mathrm{span}\{A,B\}$.
By using (\ref{bracket}) we write down the
non-vanishing brackets for the Lie algebra
$\mathfrak{g}=\mathfrak{m}+\mathfrak{h}$:
\[
\begin{array}
[c]{l}
\lbrack A,X_{1}]=-[A,X_{2}]=X_{3},\\
\lbrack B,X_{1}]=-[B,X_{2}]=X_{4},\\
\lbrack A,X_{3}]=[B,X_{4}]=X_{1}+X_{2},\\
\lbrack X_{1},X_{3}]=-[X_{2},X_{3}]
=-cX_{4}+\tilde{R}_{133}^{2}A
+\tilde{R}_{143}^{2}B,\\
\lbrack X_{1},X_{4}]=-[X_{2},X_{4}]
=cX_{3}+\tilde{R}_{143}^{2}A
+\tilde{R}_{144}^{2}B,\\
\lbrack X_{3},X_{4}]
=c\left(  X_{1}+X_{2}\right)  .
\end{array}
\]

\item If
$(\tilde{R}_{143}^{2})^{2}
=\tilde{R}_{133}^{2}\tilde{R}_{144}^{2}$
then the dimension of
$\mathfrak{h}
=
\mathrm{span}\{\tilde{R}(X,Y)|X,Y\in \mathfrak{m}\}$ is one.
As we are supposing that there exists $X,Y$ such that
$\tilde{R}(X,Y)=A$, we have:
$\tilde{R}_{144}^{2}=\tilde{R}_{143}^{2}=0$.
By using (\ref{bracket}) we write down the
non-vanishing brackets for the Lie
algebra $\mathfrak{g}=\mathfrak{m}+\mathfrak{h}$:
\[
\begin{array}
[c]{l}
\lbrack A,X_{1}]=-[A,X_{2}]=X_{3},\\
\lbrack A,X_{3}]=X_{1}+X_{2},\\
\lbrack X_{1},X_{3}]=-[X_{2},X_{3}]
=-cX_{4}+\tilde{R}_{133}^{2}A,\\
\lbrack X_{1},X_{4}]=-[X_{2},X_{4}]
=cX_{3},\\
\lbrack X_{3},X_{4}]
=c\left(  X_{1}+X_{2}\right)  .
\end{array}
\]
This corresponds to the case $\beta=\delta=0$
of the family in Theorem \ref{Loren}-2.
\end{itemize}

\subsection{Case b) \label{Case b}}

Suppose that a curvature transformation
$A=\tilde{R}(X,Y)$ exists such that
\begin{equation}
A=\alpha A_{2}+\beta A_{3},
\label{A_Lorentz_case_b}
\end{equation}
as in Proposition \ref{Lorentz}.
If $\alpha\beta\neq0$, by applying
$A\cdot\tilde{T}=0$ to $X_i,X_j$
for $1\leq i<j\leq4$, and taking
(\ref{TildaLorentz}) and (\ref{A_Lorentz_case_b})
into account we easily
obtain $\tilde{T}=0$ and hence $M$ is symmetric.
We now assume that
$\alpha\beta=0$ and $A\neq0$.

In the case $\beta=0$, $\alpha\neq0$,
by applying $A_{2}\cdot\tilde{T}=0$ to
$X_i,X_j$ for $1\leq i<j\leq4$, and
taking (\ref{TildaLorentz}) into
account we have $c=d=0$. Then, if $a\neq0$
(resp. $a=0$, $b\neq0$) we take
$W=\mathrm{span}\{X_1,X_2,X_3+\tfrac{b}{a}X_4\}$
(resp.
$W=\mathrm{span}\{X_1,X_2,X_4\}$)
and we conclude that $M$ is
decomposable by virtue of Proposition \ref{Decomposition}.

We now assume that $\alpha=0$, $\beta\neq0$.
By applying $A_{3}\cdot\tilde{T}=0$ to
$X_i,X_j$ for $1\leq i<j\leq4$,
and taking (\ref{TildaLorentz})
into account we get $a=b=0$.
If $c^{2}-d^{2}\neq0$ we take
$W=\mathrm{span}\{X_{3},X_{4},-cX_{1}+dX_{2}\}$
and we conclude that $M$ is decomposable by
virtue of Proposition \ref{Decomposition}.

We thus consider $d=\eta c\neq0$ with $\eta=\pm1$.
In this case, from
straightforward--but rather long--computations,
we get
\[
\left\{
\begin{array}
[c]{l}
\tilde{R}(X_{1},X_{3})
=\tilde{R}_{133}^{1}M+\tilde{R}_{143}^{1}N
+\tilde{R}_{134}^{3}A_{3},\\
\tilde{R}(X_{1},X_{4})
=\tilde{R}_{143}^{1}M+\tilde{R}_{144}^{1}N
+\tilde{R}_{144}^{3}A_{3},\\
\tilde{R}(X_{3},X_{4})
=-\tilde{R}_{134}^{3}M-\tilde{R}_{144}^{3}N
+\tilde{R}_{344}^{3}A_{3},\\
\tilde{R}(X_1,X_2)=0,\\
\tilde{R}(X_2,X_i)=\eta\tilde{R}(X_1,X_i),
\text{ }i=3,4,
\end{array}
\right.
\]
where
\begin{equation}
M=\left(
\begin{array}
[c]{cccc}
0 & 0 & 1 & 0\\
0 & 0 & -\eta & 0\\
1 & \eta & 0 & 0\\
0 & 0 & 0 & 0
\end{array}
\right)  \text{ and }N=\left(
\begin{array}
[c]{cccc}
0 & 0 & 0 & 1\\
0 & 0 & 0 & -\eta\\
0 & 0 & 0 & 0\\
1 & \eta & 0 & 0
\end{array}
\right)  .
\label{MN}
\end{equation}
By applying $A_3\cdot\tilde{R}=0$ to
$\left(  X_1,X_3,X_3\right)  $ we
obtain $\tilde{R}_{143}^1=\tilde{R}_{144}^3=0$.
With the choice $\left( X_1,X_3,X_4\right)  $,
we get
$\tilde{R}_{144}^{1}=\tilde{R}_{133}^{1}$,
and with choice $\left(  X_1,X_4,X_4\right) $
we get $\tilde{R}_{134}^{3}=0$.
One can verify that these four relations
($\tilde{R}_{143}^{1}=\tilde{R}_{144}^{3}=\tilde{R}_{134}^{3}=0$,
$\tilde{R}_{144}^{1}=\tilde{R}_{133}^{1}$)
are equivalent to $A_{3}\cdot\tilde{R}=0$. If
$\tilde{R}_{133}^{1}\neq0$, then
$\tilde{R}(X_{1},X_{3})$ is a skew-symmetric
endomorphism like in Proposition \ref{Lorentz} -(a)
that is already studied in
\S \ref{Case a}. Therefore, with
$\tilde{R}_{133}^{1}=0$ and
$\tilde{R}_{344}^{3}\neq0$, we have
$\mathfrak{g}=\mathfrak{m}+\mathfrak{h}$, with
$\mathfrak{h}=\mathrm{span}\{A_{3}\}$.
By using (\ref{bracket}) we obtain that
the non-vanishing brackets are
\[
\begin{array}
[c]{ll}
\lbrack A_3,X_3]=-X_4, & [A_3,X_4]=X_3,\\
\lbrack X_1,X_3]=-cX_4, & [X_1,X_4]=cX_3,\\
\lbrack X_2,X_3]=-dX_4, & [X_2,X_4]=dX_3,\\
\lbrack X_3,X_4]=cX_1-dX_2+\tilde{R}_{344}^{3}A_3. &
\end{array}
\]
Letting
\[
T_1=X_1-cA_3,
\quad
T_2=X_2-dA_3,
\quad
Y_1=cT_1-dT_2+\lambda A_3
=cX_1-dX_2+\lambda A_3,
\]
with $\lambda=\tilde{R}_{344}^{3}$ then,
since $\lambda\neq0$ and $d=\eta c\neq0$,
a basis of the same Lie algebra is given by
$\{Y_1,X_3,X_4,T_1,T_2\}$ and the only
non-null brackets are
\[
[X_3,Y_1]=\lambda X_4,
\quad
[Y_1,X_4]=\lambda X_3,
\quad
[X_3,X_4]=Y_1,
\]
that is,
$\mathrm{span}\{Y_1,X_3,X_4\}\simeq\mathfrak{sl}(2,\mathbb{R})$,
and the Lie algebra $\mathfrak{g}$ is the direct
sum of the Abelian Lie algebra
$\mathrm{span}\{T_{1},T_{2}\}\simeq\mathbb{R}^{2}$ and
$\mathfrak{sl}(2,\mathbb{R})$. The corresponding
simply connected Lie group is thus the
direct product
$\widetilde{SL(2,\mathbb{R})}\times\mathbb{R}^{2}$.

\section{Example: oscillator}

One of the most celebrated examples
of Lorentzian naturally reductive spaces
is the oscillator group. We refer to \cite{Gadea1}
for notation and definitions. This group
is defined as
$G=\mathbb{R}\times\mathbb{C}\times\mathbb{R}$
with group structure
\[
(p_1,z_1,q_1)\cdot(p_2,z_2, q_2)=(p_1+p_2
+\tfrac{1}{2}\operatorname{Im}(\bar{z}_{1}e^{iq_{1}}z_{2}),
z_1+e^{iq_1}z_2,q_1+q_2).
\]
The corresponding Lie $\mathfrak{g}$ with basis
$\mathcal{B}=(P,X,Y,Q)$ has non-vanishing brackets
\[
[X,Y]=P,\quad
[Q,X]=Y,\quad
[Q,Y]=-X.
\]
Furthermore, $G$ in endowed
with the left invariant metric
\[
\left(
\begin{array}
[c]{cccc}
\varepsilon & 0 & 0 & 1\\
0 & 1 & 0 & 0\\
0 & 0 & 1 & 0\\
1 & 0 & 0 & \varepsilon
\end{array}
\right)
\]
with respect to the basis $\mathcal{B}$ and for
$-1<\varepsilon<1$. Note that $g$ is not, a priori, a product of
metrics. The manifold is symmetric if and only if $\varepsilon=0$.
For $\varepsilon\neq0$, the naturally reductive structure tensors
$D$ and the curvature operators $\tilde{R}$ are given in
\cite{Gadea1}. With respect to the orthonormal basis
$(P^{\prime},X,Y,Q^{\prime})$,
$P^{\prime}=(2-2\varepsilon)^{-1/2}(P-Q)$,
$Q^{\prime}=(2+2\varepsilon)^{-1/2}(P+Q)$, we have
$\tilde{R}_{XY}P^{\prime}=\tilde{R}_{XY}Q^{\prime}=0$,
$\tilde{R}_{XY}X=-\varepsilon Y$, $\tilde{R}_{XY}Y=3\varepsilon
X$, and $\tilde{T}(X,Y)
=-\frac{1}{2}\sqrt{2-2\varepsilon}P^{\prime}
-\frac{1}{2}\sqrt{2+2\varepsilon}Q^{\prime}\,\ $ which is as in \S
\ref{Case b} with $c^2-d^2\neq0$. Hence $G$ must be the
semi-Riemannian product of two naturally reductive spaces with
infinitesimal decomposition $T_{e}G=W\oplus W^{\perp}$,
$W=\mathrm{span}\{X,Y,\tilde{T}(X,Y)\}$. From this, we easily get
the splitting $G=M_1\times M_2$, with
\begin{align*}
M_1  &
=\{(\lambda,0,0,-\varepsilon\lambda)~|~
\lambda\in
\mathbb{R}\}\simeq\mathbb{R},\\
M_2  &
=\{(p,x,y,0)~|~p,x,y\in\mathbb{R\}\simeq R}^{3}.
\end{align*}
Using coordinates $(\lambda;p,x,y)$ in
$M_1\times M_2$, one can check that
the matrix of $g$ in this system reads
\[
\left(
\begin{array}
[c]{cccc}
\varepsilon(\varepsilon-1)(\varepsilon+1) & 0 & 0 & 0
\\ 0 & \varepsilon & \tfrac{1}{2}\varepsilon y &
-\tfrac{1}{2}\varepsilon x
\\ 0 & \tfrac{1}{2}\varepsilon y & 1
+\tfrac{1}{4}\varepsilon y^{2} &
-\tfrac{1}{2}\varepsilon xy
\\ 0 & -\tfrac{1}{2}\varepsilon x &
-\tfrac{1}{2}\varepsilon xy & 1
+\tfrac{1}{4}\varepsilon x^{2}
\end{array}
\right)
\]
which proves that
$(G,g_{\varepsilon})=(M_{1},g_{1})\times(M_{2},g_{2})$,
with
$g_1$ Riemannian and $g_2$ Lorentzian for $-1<\varepsilon<0$
and the opposite for $0<\varepsilon<1$.

\section{Proof of Theorem \ref{dosdos}}

\subsection{Reducible cases}

Let $\left(  X_1,X_2,X_3,X_4\right) $ be a basis in $T_oM$ such
that $\langle X_i,X_j\rangle =\varepsilon_i\delta_{ij}$, with
$\varepsilon_{1}=\varepsilon_{2}=-1$,
$\varepsilon_{3}=\varepsilon_{4}=1$. From
(\ref{NaturallyReductiveCondition}) we have:
\[
0=\tilde{T}_{ij}^k\varepsilon_k
+\tilde{T}_{ik}^j\varepsilon_{j},
\quad i,j,k=1,\ldots,4,
\]
where $\tilde{T}_{ij}^{k}$ is introduced
in (\ref{T_tilda}). Therefore,
$\tilde{T}_{ij}^{i}=0$ and
$\tilde{T}_{ji}^{i}=0$,
$1\leq i\leq j\leq4$, and
by denoting
$\tilde{T}_{12}^{3}=a$,
$\tilde{T}_{12}^{4}=b$,
$\tilde{T}_{13}^{4}=c$,
$\tilde{T}_{23}^{4}=d$, we get:
\begin{equation}
\left\{
\begin{array}
[c]{l}
\tilde{T}(X_1,X_2)=aX_3+bX_4,\\
\tilde{T}(X_1,X_3)=aX_2+cX_4,\\
\tilde{T}(X_1,X_4)=bX_2-cX_3,\\
\tilde{T}(X_2,X_3)=-aX_1+dX_4,\\
\tilde{T}(X_2,X_4)=-bX_1-dX_3,\\
\tilde{T}(X_3,X_4)=-cX_1-dX_2.
\end{array}
\right.
\label{TtildaNeutral_orthonormal}
\end{equation}
As in \S \ref{Loren}, we assume
$\tilde{T}\neq0$ and there exist
$X,Y\in T_{o}M$ such that
$A=\tilde{R}(X,Y)\neq0$ so that we can apply the
classification of Proposition \ref{Neutral}.

\subsubsection{Case a1)\label{Case a 1}}

Suppose that a curvature transformation
$A=\tilde{R}(X,Y)$ exists such that
\begin{equation}
AX_{1}=0,
\quad AX_{2}=X_{4},
\quad AX_{3}=-X_{4},
\quad AX_{4}=X_{2}+X_{3}.
\label{A_Neutral_case_a1}
\end{equation}
By applying $A\cdot\tilde{T}=0$ to
$X_i,X_j$ for $1\leq i<j\leq4$, and
taking (\ref{TtildaNeutral_orthonormal})
and (\ref{A_Neutral_case_a1}) into
account we get $a=0$, $c=-b$.
As we are considering $\tilde{T}\neq0$ we have
$b^2+d^2\neq0$.

In the case, $b=0$, $d\neq0$ (resp. $b\cdot d\neq0$)
if we take $W=\mathrm{span}\{X_2,X_3,X_4\}$
(resp. $W=\mathrm{span}\{bX_1+dX_3,-bX_1+dX_2,X_4\}$)
we conclude that $M$ is decomposable by
virtue of Proposition \ref{Decomposition}.

For the case $b\neq0$, $d=0$,
(\ref{Bianchi1}), (\ref{Bianchi2}) we have
\[
\left\{
\begin{array}
[c]{l}
\tilde{R}(X_1,X_4)=\tilde{R}(X_2,X_3)=0,\\
\tilde{R}(X_1,X_3)=-\tilde{R}(X_1,X_2),\\
\tilde{R}(X_2,X_4)=-\tilde{R}(X_3,X_4),\\
\tilde{R}(X_1,X_2)=-\tilde{R}_{134}^{3}A_1-\tilde{R}_{133}^{1}B,\\
\tilde{R}(X_3,X_4)=\tilde{R}_{344}^{3}A_1-\tilde{R}_{134}^{3}B,
\end{array}
\right.
\]
with $A_1$ as in Proposition \ref{Neutral} and
\[
B=\left(
\begin{array}
[c]{cccc}
0 & -1 & 1 & 0\\
1 & 0 & 0 & 0\\
1 & 0 & 0 & 0\\
0 & 0 & 0 & 0
\end{array}
\right)  .
\]
We have two possibilities:

\begin{itemize}
\item If
$(\tilde{R}_{134}^{3})^{2}\neq-\tilde{R}_{133}^{1}\tilde{R}_{344}^{3}$
then
$\mathfrak{h}=\mathrm{span}\{A_{1},B\}$ and by using
(\ref{bracket}) the non-vanishing brackets of the Lie algebra
$\mathfrak{g}=\mathfrak{m}+\mathfrak{h}$ are
\[
\begin{array}
[c]{l}
\lbrack A_1,X_2]=-[A_1,X_3]=X_4,\\
\lbrack B,X_1]=[A_1,X_4]=X_2+X_3,\\
\lbrack B,X_2]=-[B,X_3]=-X_1,\\
\lbrack X_1,X_2]=-[X_1,X_3] =-bX_4-\tilde{R}_{134}^3A_1
-\tilde{R}_{133}^1B,\\
\lbrack X_2,X_4]=-[X_3,X_4] =bX_1-\tilde{R}_{344}^3A_1
+\tilde{R}_{134}^3B,\\
\lbrack X_1,X_4]=-b\left( X_2+X_3\right) .
\end{array}
\]

\item If
$(\tilde{R}_{134}^{3})^{2}
=-\tilde{R}_{133}^{1}\tilde{R}_{344}^{3}$,
the dimension of
$\mathfrak{h}
=\mathrm{span}\{\tilde{R}(X,Y)|X,Y\in \mathfrak{m}\}$
is one. As we supposed that
there exists $X,Y$ such that
$\tilde{R}(X,Y)=A_{1}$, with
$A_{1}$ as in (\ref{A_Neutral_case_a1}),
we have that
$\tilde{R}_{133}^{1}=\tilde{R}_{134}^{3}=0$.
The non-vanishing brackets of the Lie algebra
$\mathfrak{g}=\mathfrak{m}+\mathfrak{h}$ are
\[
\begin{array}
[c]{l}
\lbrack A_1,X_2]=-[A_1,X_3]=X_4,\\
\lbrack A_1,X_4]=X_2+X_3,\\
\lbrack X_1,X_2]=-[X_1,X_3]=-bX_4,\\
\lbrack X_2,X_4]=-[X_3,X_4]=bX_1-\tilde{R}_{344}^3A_1,\\
\lbrack X_1,X_4]=-b\left( X_2+X_3\right) .
\end{array}
\]
This corresponds to the case
$\beta=\delta=0$ of the family in Theorem
\ref{dosdos}-2.
\end{itemize}

\subsubsection{Case a2)}

Suppose that a curvature transformation
$A=\tilde{R}(X,Y)$ exists such that
\begin{equation}
AX_{1}=-\alpha X_{2},\quad
AX_{2}=\alpha X_{1},\quad
AX_{3}=-\beta X_{4},\quad
AX_{4}=\beta X_{3}.
\label{A_Neutral_case_a2}
\end{equation}
If $\alpha\neq0$, by applying
$A\cdot\tilde{T}=0$ to
$X_{1},X_{2}$ and to
$X_{3},X_{4}$, and taking
(\ref{TtildaNeutral_orthonormal}) and
(\ref{A_Neutral_case_a2}) into account
we deduce: $a=b=c=d=0$. Therefore $M$
is symmetric.

For $\alpha=0$ the condition
$A\cdot\tilde{T}=0$ only gives that $a=b=0$ in
(\ref{TtildaNeutral_orthonormal}).
As we are considering $\tilde{T}\neq0$, we
have $c^2+d^2\neq0$. In that case,
if we take $W=\mathrm{span}\{cX_1+dX_2,X_3,X_4\}$
we conclude that $M$ is decomposable by virtue
of Proposition \ref{Decomposition}.

\subsubsection{Case a3)\label{case a3}}

Suppose that a curvature transformation
$A=\tilde{R}(X,Y)$ exists such that
\[
AX_1=\beta X_3,\quad
AX_2=\alpha X_4,\quad
AX_3=\beta X_1,\quad
AX_4=\alpha X_2.
\]
If $\beta\neq0$, the condition
$A\cdot\tilde{T}=0$ gives $a=b=c=d=0$.
Therefore $M$ is symmetric.

For $\beta=0$ the condition $A\cdot\tilde{T}=0$
now gives $a=c=0$ in (\ref{TtildaNeutral_orthonormal}).
As we are considering $\tilde{T}\neq0$, we
have $b^2+d^2\neq0$.

If $d^2-b^2\neq0$ and we take
$W=\mathrm{span}\{X_2,X_4,bX_1+dX_3\}$
we conclude that $M$ is decomposable
by virtue of Proposition \ref{Decomposition}.
We thus assume $d=\eta b$, $\eta=\pm1$. In this case,
from straightforward--but rather long--computations, we get
\[
\left\{
\begin{array}
[c]{l}
\tilde{R}(X_1,X_2)
=\tilde{R}_{122}^{1}M
+\tilde{R}_{142}^{1}N
+\tilde{R}_{124}^{2}A_{(\beta=0)},\\
\tilde{R}(X_1,X_4)
=\tilde{R}_{142}^{1}M
+\tilde{R}_{144}^{1}N
+\tilde{R}_{144}^{2}A_{(\beta=0)},\\
\tilde{R}(X_2,X_4)
=\tilde{R}_{124}^{2}M
+\tilde{R}_{144}^{2}N
+\tilde{R}_{244}^{2}A_{(\beta=0)},\\
\tilde{R}(X_1,X_3)=0,\\
\tilde{R}(X_3,X_j)
=-\eta\tilde{R}(X_1,X_j),
\quad j=2,4,
\end{array}
\right.
\]
where
\begin{equation*}
M=\left(
\begin{array}
[c]{cccc}
0 & 1 & 0 & 0\\
-1 & 0 & \eta & 0\\
0 & \eta & 0 & 0\\
0 & 0 & 0 & 0
\end{array}
\right)  \text{ and }N=\left(
\begin{array}
[c]{cccc}
0 & 0 & 0 & 1\\
0 & 0 & 0 & 0\\
0 & 0 & 0 & \eta\\
1 & 0 & -\eta & 0
\end{array}
\right)  .
\label{MNneutral}
\end{equation*}
By applying $A\cdot\tilde{R}=0$ to
$(X_1,X_2,X_i)$, $i=1,2$, we obtain:
$\tilde{R}_{142}^{1}=\tilde{R}_{144}^{2}=0$,
$\tilde{R}_{144}^{1}=-\tilde{R}_{122}^{1}$.
With the choice $(X_2,X_4,X_4)$ we get
$\tilde{R}_{124}^{2}=0$. One can easily
prove that with these four conditions
($\tilde{R}_{142}^{1}=\tilde{R}_{144}^{2}
=\tilde{R}_{124}^{2}=0$,
$\tilde{R}_{144}^{1}=-\tilde{R}_{122}^{1}$)
are equivalent to the condition
$A\cdot\tilde{R}=0$. Furthermore,
from condition
$\tilde{R}(X_1,X_4)\cdot\tilde{R}=0$
we obtain
$\tilde{R}_{122}^{1}\tilde{R}_{244}^{2}=0$.
As we assume that the manifold is non symmetric,
either
$\tilde{R}_{244}^{2}=0$,
$\tilde{R}_{122}^{1}\neq0$ or
$\tilde{R}_{244}^{2}\neq0$,
$\tilde{R}_{122}^{1}= 0$.
The former case has been already studied
in \S \ref{Case a 1}, for $M$ is a matrix of type a1.
Then $\tilde{R}_{244}^{2}\neq0$ and
$\mathfrak{h}=\mathrm{span}\{A_{3}\}$
as in Proposition \ref{Neutral} with $\beta=0$,
$\alpha=1$. By using (\ref{bracket})
the non-vanishing brackets of the Lie
algebra
$\mathfrak{g}=\mathfrak{m}+\mathfrak{h}$ are
\[
\begin{array}
[c]{ll}
\lbrack A_3,X_2]=X_4, & [A_3,X_4]=X_2,\\
\lbrack X_1,X_2]=-bX_4, & [X_1,X_4]=-bX_2,\\
\lbrack X_2,X_3]=-dX_4, & [X_3,X_4]=dX_2,\\
\lbrack X_2,X_4]=bX_1+dX_3+\tilde{R}_{244}^2A_3. &
\end{array}
\]
Letting $T_1=X_1+bA_3$,
$T_2=X_3-dA_3$,
$Y_1=bX_1+dX_3+\lambda A_3$,
with $\lambda=\tilde{R}_{244}^{2}$,
since $\lambda \neq0$,
a basis of the same Lie algebra is given by
$(T_1,T_2,Y_1,X_2,X_4)$
and the only non-null brackets are
\[
[Y_1,X_2]=\lambda X_4,\quad
[Y_1,X_4]=\lambda X_2,\quad
[X_2,X_4]=Y_1,
\]
that is, $Y_1,X_2,X_4$ generate
$\mathfrak{sl}(2,\mathbb{R})$, and
$\mathfrak{g}$ is the direct sum of the
$2$-dimensional Abelian Lie algebra
$\mathrm{span}\{T_1,T_2\}$ and
$\mathfrak{sl}(2,\mathbb{R})$. The
corresponding simply connected Lie
group is thus the direct product
$\widetilde{SL(2,\mathbb{R})}\times\mathbb{R}^{2}$.

\subsection{Irreducible cases}

For a basis $\left(  X_1,X_2,X_3,X_4\right) $ of
$T_oM$ such that
\[
\langle X_2,X_3\rangle =-\langle X_1,X_4 \rangle =1\text{ and
}\langle X_i,X_j\rangle =0, \quad \text{otherwise,}
\]
condition (\ref{NaturallyReductiveCondition}) gives
\begin{equation}
\left\{
\begin{array}
[c]{l}
\tilde{T}(X_1,X_2)=cX_1-aX_2,\\
\tilde{T}(X_1,X_3)=dX_1+aX_3,\\
\tilde{T}(X_1,X_4)=dX_2+cX_3,\\
\tilde{T}(X_2,X_3)=-bX_1+aX_4,\\
\tilde{T}(X_2,X_4)=-bX_2+cX_4,\\
\tilde{T}(X_3,X_4)=bX_3+dX_4.
\end{array}
\right.
\label{TtildaNeutral2}
\end{equation}

\subsubsection{Case b1)}

Suppose that a curvature transformation
$A=\tilde{R}(X,Y)$ exists such that
\begin{equation}
AX_1=\nu X_2,\quad
AX_2=-\nu X_1,\quad
AX_3=X_1+\nu X_4,\quad
AX_4=X_2-\nu X_3.
\label{A_Neutral_case_b1}
\end{equation}
If $\nu\neq0$, by applying $A\cdot\tilde{T}=0$ to $X_i,X_j$,
$1\leq i<j\leq4$, and taking (\ref{TtildaNeutral2}) and
(\ref{A_Neutral_case_b1}) into account we get $a=b=c=d=0$, hence
$\tilde{T}=0$ and therefore $M$ is symmetric. If $\nu=0$, we just
get $a=c=0$. As we are considering $\tilde{T}\neq0$, at least one
of $b$ and $d$ is different to $0$. We can assume that $b\neq0$
(if $b=0$, the new basis $(-X_2,X_1,-X_4,X_3)$ preserves the
metric and the expression of $A$ but switches $b$ to $d$). In this
case from the Bianchi identities (\ref{Bianchi1}),
(\ref{Bianchi2}) and imposing $A\cdot\widetilde{R}=0$, we obtain:
\[
\left\{
\begin{array}
[c]{l}
\tilde{R}(X_1,X_2)=0,\\
\tilde{R}(X_1,X_i)
=-\frac{d}{b}\tilde{R}(X_2,X_i),\quad i=3,4,\\
\tilde{R}(X_2,X_3)
=-\frac{d}{b}\tilde{R}(X_2,X_4),\\
\tilde{R}(X_2,X_4)
=\tilde{R}_{244}^{3}B+\tilde{R}_{344}^{3}B_1,\\
\tilde{R}(X_3,X_4)
=\tilde{R}_{344}^{3}B+\tilde{R}_{344}^{2}B_1,
\end{array}
\right.
\]
where $B_1$ is as in
Proposition \ref{Neutral} with $\nu=0$ and
\begin{equation}
B=\left(
\begin{array}
[c]{cccc}
-\frac{d}{b} & 1 & 0 & 0\\
-\frac{d^{2}}{b^{2}} & \frac{d}{b} & 0 & 0\\
0 & 0 & -\frac{d}{b} & 1\\
0 & 0 & -\frac{d^{2}}{b^{2}} & \frac{d}{b}
\end{array}
\right)  .
\label{B}
\end{equation}
One can check that both $B\pm B_{1}$
are reducible matrix equivalent to
$A_{1}$ in Proposition \ref{Neutral}
(note that
$\mathrm{span}\{\mp X_2+X_3+\frac{d}{b}X_4\}$
is an invariant and non-degenerate subspace).
This means that
$\mathfrak{h}=\mathrm{span}\{B,B_{1}\}$
is also generated by the two reducible
endomorphisms $B^{\prime}=B+B_{1}$ and
$B^{\prime\prime }=B-B_{2}$.
This case has been already studied in \ref{Case a 1}.
This implies that $\mathfrak{g}$ and
$\mathfrak{h}$ must be as in case \ref{Case a 1} above.

\subsubsection{Case b2)}

Suppose that a curvature transformation
$A=\tilde{R}(X,Y)$ exists such that
\begin{equation}
AX_{1}=\lambda X_{1},\quad
AX_{2}=-\lambda X_{2},\quad
AX_{3}=X_{1}+\lambda X_{3}, \quad
AX_{4}=X_{2}-\lambda X_{4},
\label{A_Neutral_case_b2}
\end{equation}
with $\lambda\neq0$. By applying
$A\cdot\tilde{T}=0$ to $X_1,X_2$ and to
$X_3,X_4$, and taking (\ref{TtildaNeutral2})
and (\ref{A_Neutral_case_b2})
into account we deduce: $a=b=c=d=0$,
hence $\tilde{T}=0$ and therefore $M$ is symmetric.

\subsubsection{Case b3)}

Suppose that a curvature transformation
$A=\tilde{R}(X,Y)$ exists such that
\begin{equation}
\left\{
\begin{array}
[c]{ll}
AX_{1}=\xi X_{2}+\nu X_{4}, &
AX_{2}=\xi X_{1}+\nu X_{3},\\
AX_{3}=-\nu X_{3}+\xi X_{4}, &
AX_{4}=-\nu X_{1}+\xi X_{3},
\end{array}
\right.
\label{A_Neutral_case_b4}
\end{equation}
with $\xi\cdot\nu\neq0$. By applying
$A\cdot\tilde{T}=0$ to $X_{i},X_{j}$ for
$1\leq i<j\leq4$, and taking (\ref{TtildaNeutral2}) and
(\ref{A_Neutral_case_b4}) into account we obtain:
$a=b=c=d=0$ and therefore
$M$ is symmetric.

\section{Study of the new manifolds}

In this section we analyze the geometry of the manifolds given in
Theorems \ref{Loren} and \ref{dosdos}. We prove that in the
generic case they are not symmetric nor decomposable. For that
purpose, we need the computation of the covariant derivative of
the curvature tensor and the holonomy of the Levi-Civita
connection. With respect to the former, from (\ref{2}) and
(\ref{D}) we have
\begin{align}
(\nabla_{X}R)\left(  Y,Z\right)  W &
=-\tfrac{1}{2}\tilde{T}(X,R(Y,Z)W)
+\tfrac{1}{2}R(\tilde{T}(X,Y),Z)W
\label{nablaR}\\
&
+\tfrac{1}{2}R(Y,\tilde{T}(X,Z))W
+\tfrac{1}{2}R(Y,Z)\tilde{T}(X,W).
\nonumber
\end{align}
For the latter, we recall that (see \cite[X, Corollary 4.5]{KN})
the holonomy algebra of a reductive homogeneous manifold $M=G/H$,
with $H$ being the isotropy of $o\in M$, is the smallest
subalgebra $\mathfrak{hol}\subset so(\mathfrak{m},g_o)$ containing
the $R(X,Y)_o$, $X,Y\in\mathfrak{m}$, such that
$[\Lambda_{\mathfrak{m}}(X),\mathfrak{hol}]
\subset\mathfrak{hol}$, for all $X\in\mathfrak{m}$, where
$\Lambda_{\mathfrak{m}}(X) \colon\mathfrak{m}\rightarrow
\mathfrak{m}$ is $\Lambda_{\mathfrak{m}}(X)(Y)
=\frac{1}{2}[X,Y]_{\mathfrak{m}}$. Note that if the holonomy
algebra does not posses any proper non-degenerate invariant
subspace, the manifold must be indecomposable. We then have the
following results.

\begin{proposition}
\label{Case 1}The Lorentzian manifold $G/H$ in Theorem
\ref{Loren}-2, with $c\neq 0$ and $\alpha\delta-\beta^2\neq 0$, is
flat if and only if $\beta=0$ and
$\alpha=\delta=\tfrac{1}{4}c^{2}$. Otherwise it is indecomposable.
Furthermore, it is symmetric if and only if $\beta=0$ and
$\alpha=\delta$.

The subalgebra $\mathrm{span}\{X_1,X_2,X_3,X_4,A\}$ for $\beta
=\delta=0$ which corresponds to \ref{Case a},
$\beta^2=\alpha\delta$, is also non-symmetric and indecomposable.
\end{proposition}

\begin{proof}
Taking (\ref{R}) and (\ref{D}) into account we get
\[
\left\{
\begin{array}
[c]{ll}
R(X_1,X_2)=R(X_3,X_4)=0,\quad & R(X_1,X_3)=(\alpha-\tfrac{1}{4}c^{2})A+\beta B, \\
R(X_1,X_4)=\beta A+(\delta-\tfrac{1}{4}c^{2})B,\quad &
R(X_2,X_j)=-R(X_1,X_j), \quad j=3,4,
\end{array}
\right.
\]
and we get the condition about flatness. For the covariant
derivative of the curvature, from (\ref{nablaR}) we have
$(\nabla_{X_1}R)\left( X_1,X_3\right) X_1=c\beta X_3+\tfrac{1}{2}c
\left( \delta-\alpha\right) X_4,$ which only vanishes for
$\beta=0$ and $\alpha=\delta$. In that case, it is easy to see
that $(\nabla _{X_i}R)(X_j,X_k)X_l=0$, for all $i,j,k,l$ and $M$
is locally symmetric.

If
$(\alpha-\tfrac{1}{4}c^2)(\delta-\tfrac{1}{4}c^2)-\beta^2\neq0$,
in view of the expressions of $R(X_1,X_3)$ and $R(X_1,X_4)$ above,
condition $[\Lambda_{\mathfrak{m}}(X),\mathfrak{hol}]
\subset\mathfrak{hol}$ gives
$\mathfrak{hol}=\mathrm{span}\{A,B\}$. These matrices do not have
any common invariant non-degenerate subspace and therefore $M$ is
irreducible.

If
$(\alpha-\tfrac{1}{4}c^2)(\delta-\tfrac{1}{4}c^2)-\beta^2=0$,
then $R(X,Y)_{o}$ is generated by a single element,
for instance $(\alpha-\tfrac
{1}{4}c^{2})A+\beta B$. On the other hand,
one can check that
$[\Lambda_{\mathfrak{m}}(X_1),A]=-\tfrac{c}{2}B$,
$[\Lambda_{\mathfrak{m}}(X_1),B]=\tfrac{c}{2}A$.
Then condition
$[\Lambda_{\mathfrak{m}}(X_1),\mathfrak{hol}]
\subset\mathfrak{hol}$
gives again $\mathfrak{hol}=\mathrm{span}\{A,B\}$
unless $\alpha-\tfrac{1}{4}c^2=\delta-\tfrac{1}{4}c^2=\beta=0$.

The study of the subalgebra
$\mathrm{span}\{X_{1},X_{2},X_{3},X_{4},A\}$ for
$\beta=\delta=0$ is done similarly.
\end{proof}

\begin{proposition}
\label{SL} The Lorentzian manifold
$(SL(2,\mathbb{R})\times\mathbb{R}^{2})/\mathbb{R}$ in Theorem
\ref{Loren}-1 defined infinitesimally by the Lie brackets
\[
\begin{array}
[c]{ll}
\lbrack A,X_3]=-X_4, &
\lbrack A,X_4]=X_3,\\
\lbrack X_1,X_3]=-cX_4, &
\lbrack X_1,X_4]=cX_3,\\
\lbrack X_2,X_3]=-\eta cX_4, &
\lbrack X_2,X_4]=\eta cX_3,\\
\lbrack X_3,X_4]=cX_1-\eta cX_2+\alpha A, &
\end{array}
\]
as in \S \ref{Case b} with $c\neq0$,
$\eta=\pm1$, $\alpha\neq0$, is not flat,
non-symmetric and indecomposable.
\end{proposition}

\begin{proof}
Taking (\ref{R}) and (\ref{D})
into account we get
\[
\left\{
\begin{array}
[c]{ll}
R(X_1,X_2)=0,\quad & R(X_1,X_3)=-\tfrac{1}{4}c^2M,\\
R(X_1,X_4)=-\tfrac{1}{4}c^2N,\quad &R(X_2,X_j)=\eta R(X_1,X_j), j=3,4,\\
R(X_3,X_4)=\alpha A,
\end{array}
\right.
\]
where $M$, $N$ are defined in (\ref{MN}). Moreover, from
(\ref{nablaR}), we have
\[
(\nabla_{X_3}R)\left(  X_1,X_3\right)
X_3=-\tfrac{1}{2}c\alpha X_4,
\]
so that the manifold is not locally symmetric.

In view of the expressions
of $R(X_i,X_j)$ above, condition
$[\Lambda_{\mathfrak{m}}(X),\mathfrak{hol}]
\subset\mathfrak{hol}$ gives
$\mathfrak{hol}=\mathrm{span}\{M,N,A\}$.
It is easy to see that these matrices do not share
any common invariant non-degenerate subspace.
\end{proof}

\begin{proposition}
\label{Case 2} The $(2,2)$-signature manifold $G/H$ in Theorem
\ref{dosdos}-2 with $b\neq0$, $\alpha^{2}\neq-\beta\delta$, in
theorem \ref{dosdos} is flat if and only if $\alpha=0$ and
$\beta=\delta=\tfrac{1}{4}b^2$. Otherwise it is indecomposable.
Furthermore, it is symmetric if and only if $\alpha=0$ and
$\beta=\delta$.

The subalgebra $\mathrm{span}\{X_1,X_2,X_3,X_4,A_1\}$
for $\beta=\alpha=0$ which corresponds to \ref{Case a 1},
$\alpha^2=-\beta \delta$,
is also non-symmetric and indecomposable.
\end{proposition}

\begin{proof}
Taking (\ref{R}) and (\ref{D}) into account we get
\[
\left\{
\begin{array}
[c]{ll} R(X_1,X_2)=-\alpha A_1 -\left( \beta-\tfrac{1}{4}b^{2}
\right)  B,\quad & R(X_2,X_4)= -\left( \delta-\tfrac{1}{4}b^2
\right)  A+\alpha B,\\
R(X_2,X_3)=R(X_1,X_4)=0,\quad & R(X_3,X_j)=-R(X_2,X_j), \quad
j=1,4,
\end{array}
\right.
\]
where $B$ is defined in (\ref{B}) and we get the condition about
flatness. From (\ref{nablaR}), we have $(\nabla_{X_2}R)\left(
X_3,X_1\right) X_2=-b\alpha X_1+\tfrac{1}{2}b \left(
\beta-\delta\right)  X_4,$ which only vanishes for $\alpha=0$ and
$\beta=\delta$. In that case, it is easy to see that all $(\nabla
_{X_i}R) (X_j,X_k)X_l=0$, for all $i,j,k,l$ and $M$ is locally
symmetric.

If $(\beta-\tfrac{1}{4}b^2)(\delta-\tfrac{1}{4}b^2)-\alpha^2 \neq
0$, in view of the expressions of $R(X_1,X_2)$ and $R(X_2,X_4)$
above, condition
$[\Lambda_{\mathfrak{m}}(X_2),\mathfrak{hol}]\subset
\mathfrak{hol}$ gives $\mathfrak{hol}=\mathrm{span}\{A,B\}$. These
matrices do not have any common invariant non-degenerate subspace
and therefore $M$ is irreducible. If
$(\beta-\tfrac{1}{4}b^2)(\delta-\tfrac{1}{4}b^2) -\alpha^2=0$,
then $R(X,Y)_{o}$ is generated by a single element, for instance
$\alpha A_1+(\beta-\tfrac{1}{4}b^2)B$. On the other hand, one can
check that $[\Lambda_{\mathfrak{m}}(X_2),A]=-\tfrac{b}{2}B$,
$[\Lambda_{\mathfrak{m}}(X_2),B]=-\tfrac{b}{2}A$. Then condition
$[\Lambda_{\mathfrak{m}}(X_2),\mathfrak{hol}]\subset\mathfrak{hol}$
gives $\mathfrak{hol}=\mathrm{span}\{A,B\}$ unless
$\beta-\tfrac{1}{4}b^{2}=\pm\alpha$. Otherwise,
$\mathfrak{hol}=\mathrm{span}\{A_1\pm B\}$ which do not have
invariant non-degenerate subspaces.

The study of the subalgebra
$\mathrm{span}\{X_1,X_2,X_3,X_4,A_1\}$
for $\beta=\alpha=0$ is done similarly.
\end{proof}

\begin{proposition}
The $(2,2)$-signature manifold
$(SL(2,\mathbb{R})\times\mathbb{R}^{2})/\mathbb{R}$
in Theorem \ref{dosdos}-1 defined infinitesimally
by the Lie brackets
\[
\begin{array}
[c]{ll}
\lbrack A,X_2]=X_4, & [A,X_4]=X_2,\\
\lbrack X_1,X_2]=-bX_4, & [X_1,X_4]=-bX_2,\\
\lbrack X_2,X_3]=-\eta bX_4, & [X_3,X_4]=\eta bX_3,\\
\lbrack X_3,X_4]=b(X_1+\eta X_3)+\alpha A, &
\end{array}
\]
as in \S \ref{case a3} with $b\neq0$,
$\eta=\pm1$, $\alpha\neq0$, in is
non-symmetric and indecomposable.
\end{proposition}

We do not include the proof as
it is similar to the one of Proposition
\ref{SL}.

{\small \noindent\textbf{Authors' addresses:} }

{\small \smallskip}

{\small \noindent W.B.:
\'{E}cole Nationale Polytechnique d'Oran,
D\'{e}par\-tement de Math\'{e}matiques et Informatique,
B.P. 1523, El M'Naouar, Oran, Algeria.
\newline
\emph{E-mail:\/} \texttt{batatwafa@yahoo.fr} }

{\small \smallskip}

{\small \noindent M.C.L.: ICMAT(CISC, UAM, UC3M,
UCM),\\
Departamento de Geometr\'\i a y Topolog\'\i a, Facultad de
Matem\'aticas, Universidad Complutense de Madrid, Plaza de
Ciencias 3, 28040--Madrid, Spain.
\newline
\emph{E-mail:\/} \texttt{mcastri@mat.ucm.es}
}

{\small \smallskip}

{\small \noindent E.R.M.:
Departamento de Matem\'{a}tica Aplicada
a la Edificaci\'{o}n, al Medio Ambiente y al Urbanismo,
E.T.S. Arquitectura, U.P.M.,
Juan de Herrera 4, 28040--Madrid, Spain.
\newline\emph{E-mail:\/}\texttt{eugenia.rosado@upm.es} }

\end{document}